\def\Longversion{1} 

\documentclass{ifacconf}

\usepackage{mathtools}
\usepackage{amssymb}
\usepackage{graphicx}
\usepackage{subcaption}
\usepackage{url}
\usepackage{natbib}        
\usepackage{multirow}

\graphicspath{ {./images/} }

\usepackage[utf8]{inputenc}

\newtheorem{theorem}{Theorem}

\newtheorem{lemma}{Lemma}

\newtheorem{assumption}{Assumption}
\newtheorem{proposition}{Proposition}
\theoremstyle{definition}

\usepackage{graphicx, color}
\graphicspath{{images/}}
\newenvironment{proof}{\medskip\noindent{\it Proof. }}{\hfill$\square$}

\usepackage{mathrsfs} 
\usepackage{enumitem}

\newcommand{\RR}{\mathbb{R}}

\newcommand{\X}{\mathcal{X}}

\newcommand{\V}{\mathcal{V}}
\newcommand{\Tinv}{\boldsymbol{T}^{\mathrm{inv}}}

\newcommand{\dd}{\mathrm{d}}

\newcommand{\OO}{\mathcal{O}}

\DeclareMathOperator{\diag}{diag}

\def \dsigma{\boldsymbol{\sigma}}
\def \dz{\boldsymbol{z}}
\def \dT{\boldsymbol{T}}
\def \dTa{\boldsymbol{T_a}}
\def \dR{\boldsymbol{R}}
\def \dphi{\boldsymbol{\phi}}
\def \dlambda{\boldsymbol{\lambda}}
\def \Zstar{\bar{Z}}
\def \dH{\boldsymbol{H}}

\def\leftrestrict{\left.\!\left[}
\def\rightrestrict{\right]\!\right|}

\begin{document}
\begin{frontmatter}
\if\Longversion1
\title{
On the existence of KKL observers with nonlinear contracting dynamics\\ (long version)
}
\else
\title{
On the existence of KKL observers with nonlinear contracting dynamics
}
\fi

\author[First]{Victor Pachy}
\author[Second]{Vincent Andrieu\thanksref{thanks1}}
\author[Third]{Pauline Bernard}
\author[Fourth]{Lucas Brivadis}
\author[Third]{Laurent Praly}

\thanks[thanks1]{The research of this author was funded in whole or in part by the National Research Agency (ANR) under the project titled "ANR-23-CE48-0006-02."}

\address[First]{Mines Paris, Université PSL, 60 boulevard Saint-Michel, Paris, France (e-mail: victor.pachy@etu.minesparis.psl.eu).}
\address[Second]{Univ. Lyon, Universit\'e Claude Bernard Lyon 1, CNRS, LAGEPP UMR 5007, 43 bd du 11 novembre 1918, F-69100 Villeurbanne, France (e-mail: vincent.andrieu@gmail.com)}
\address[Third]{Centre Automatique et Syst\`emes, Mines Paris, Universit\'e PSL, 60 boulevard Saint-Michel, Paris, France (e-mail: (firstname).(lastname)@minesparis.psl.eu)}
\address[Fourth]{Université Paris-Saclay, CNRS, CentraleSupélec, Laboratoire des Signaux et Systèmes, 91190, Gif-sur-Yvette, France (e-mail: lucas.brivadis@centralesupelec.fr)}

\begin{abstract}                
KKL (Kazantzis-Kravaris/Luenberger) observers are based on the idea of immersing a given nonlinear system into a target system that is a \textit{linear} stable filter of the measured output.
In the present paper, we extend this theory by allowing this target system to be a \textit{nonlinear} contracting filter of the output.
We prove, under a differential observability condition, the existence of these new KKL observers.
We motivate their introduction by showing numerically the possibility of combining convergence speed and robustness to noise, unlike what is known for linear filtering.
\end{abstract}


\end{frontmatter}

\section{Introduction.}

We consider nonlinear systems of the form
\begin{subequations}
    \label{eq:syst}
\begin{align}
    \dot{x} &= f(x),  \label{eq:syst_f}\\
    y &=h(x), \label{eq:syst_h}
\end{align}
\end{subequations}
where $x$ lying in $\RR^n$ is the state of the system,
$y$ lying in $\RR$ is the measured output,
and $f$ and $h$ are smooth maps. 

The synthesis of state observers for such systems is a major topic in control theory, and numerous methods have been developed over the years. Interested readers can refer to \cite{BERNARD2022224} for an overview of existing methods.                             
Among these methods, the so-called KKL approach has recently garnered significant attention due to its generality and the weak assumptions required ensuring its existence.
This methodology, originating from the seminal work of \cite{luenberger1964observing} for linear systems, has been extended for nonlinear systems first locally in \cite{kazantzis1998nonlinear, Shoshitaishvili_TSA_90}
and then globally in
\cite{KreisselmeierEngel_TAC_03,andrieu2006, brivadis_further}.
The strategy consists in finding a positive integer $m$ and a nonlinear mapping $\dT:\RR^n\to\RR^m$ such that if $(x, y)$ is a solution of \eqref{eq:syst}, then $\dz=\dT(x)$ satisfies
\begin{equation}
\label{eq:zfilter_lin}
    \dot \dz = A\dz + By,
\end{equation}
where $A\in\RR^{n\times m}$ is a Hurwitz matrix and $B\in\RR^m$ is a vector.
In that case, $\frac{\dd}{\dd t} (\dz-\dT(x)) = A(\dz-\dT(x))$, hence an asymptotic exponential approximation of $\dT(x)$ can be obtained by running the linear filter of the output \eqref{eq:zfilter_lin}. 
Then, if $\dT$ admits a uniformly continuous left-inverse $\Tinv:\RR^m\to\RR^n$, an observer of $x$ can be defined as $\hat x = \Tinv(z)$ and on has  $\lim_{t\to+\infty}(\hat x(t) - x(t))=0$.

Different kind of sufficient conditions for the existence of such mappings $\dT$ and $\Tinv$ are given in \cite{andrieu2006}.

In this paper, we propose to extend the class of admissible filters of the output for the design of the observer. More precisely, we replace \eqref{eq:zfilter_lin} by
\begin{equation}
\label{eq:zfilter_nonlin}
    \dot \dz = k \dsigma(\dz, y)
\end{equation}
where $k$ is a positive real number and $\dsigma$ ensures that \eqref{eq:zfilter_nonlin} has \textit{exponentially contracting} dynamics (in a sense to be defined), ensuring in particular that the distance between any pair of solutions sharing the same $y$ exponentially decreases towards $0$.
This is a natural extension, since, similar to the linear case, if one finds a mapping $T$ such that $\dz=\dT(x)$ is solution to \eqref{eq:zfilter_nonlin}, an asymptotic exponential estimation of $\dT(x)$ can be obtained, by running \eqref{eq:zfilter_nonlin} from any initial condition.
Then, as with the linear filter \eqref{eq:zfilter_lin}, if $\dT$ admits a 
left-inverse $\Tinv:\RR^m\to\RR^n$, an observer can be defined as
\begin{equation}
\label{eq:xhat}
    \hat x = \Tinv(\dz)
\end{equation}
and one has $\lim_{t\to+\infty}(\hat x(t) - x(t))=0$ provided that $\Tinv$ is uniformly continuous.

This study is motivated by three main reasons: nonlinear filters may (i) give access to better observer performance, for instance allowing to combine convergence speed and robustness to noise or increase robustness to model uncertainties; (ii) bridge the gap between the fields of observer design and recursive neural networks in machine learning; (iii) give more flexibility in the research of an analytical expression of the map $\dT$. 

Regarding item (i), while it is well-known that in the context of linear dynamics with Gaussian noise, the use of linear correction gains, as provided by a Kalman filter, is optimal (\cite{KalBuc}), the same cannot be said when departing from this context.  For instance, it is typically interesting to introduce nonlinear phenomena such as saturations in the case of sporadic noise, or, dead zones in the case of high-frequency and low-amplitude noise \cite{tarbouriech2022lmi}.
Furthermore, the interest in non-linear observers for achieving robustness to model errors is typically highlighted in the context of homogeneous observers or sliding mode observers
 (see for instance \cite{levant2003higher}).
 Moreover, convergence speed and robustness to noise are usually antinomic, unless varying/switching gains are introduced, as in \cite{sanfelicepraly2011}, \cite{boizotadaptive} and literature therein, or switches among different observers, depending on the size of the output error, as in \cite{petri2023hybrid,chong2015parameter,esfandiari2019bank}, but with thresholds or parameters that may be hard to tune. Nonlinear correction terms ensuring fast convergence while limiting the impact of the noise have also been proposed for triangular forms in \cite{prasovkhalil2013}.
Our goal in considering \eqref{eq:zfilter_nonlin} is to allow such performance without relying on a canonical form for the system \eqref{eq:syst}, or having any online thresholds to handle.

This also brings us to item (ii) since it is well-known in the machine learning community that the addition of nonlinear terms can enhance the expressiveness of neural networks and improve performance. Actually the resemblance between KKL observers -- where the system output is fed to a series of filters and the estimate is recovered from those internal states via a nonlinear map to be found -- and recursive neural networks (RNNs) is uncanny (see \cite{janny2021deep} for a comparison). In fact, when no explicit expression of the map $\dT$ is available, it has been proposed to learn an approximate model of it via neural networks in \cite{RamDiMMorSilBer,NiaCaoSunDasJoh,BuiBahDiM,PerNad,janny2021deep}. Allowing nonlinear contractions in KKL could pave the road towards understanding this similarity and providing theoretical foundations to the convergence of RNNs as well as guidelines in terms of dimensions.

However, all those methods require a significant (offline) computational load and suffer from a curse of dimensionality, sometimes at the expense of asymptotic precision on the online estimate. Our third goal (iii) is thus to enlarge the class of systems for which an \textit{explicit} expression of the transformation $\dT$ can be obtained by providing more flexibility in the target dynamics \eqref{eq:zfilter_nonlin}.

From a theoretical point of view, the \textit{existence} of an immersion $\dT$ of any system \eqref{eq:syst} into dynamics \eqref{eq:zfilter_nonlin} is typically guaranteed, for instance exploiting the theory of uniformly convergent systems in \cite{pavlov_uniformRegulation}. The challenge rather lies in proving its \textit{injectivity} under observability conditions. This paper achieves a first step by showing this injectivity for a particular structure of \eqref{eq:zfilter_nonlin}, consisting of $m$ sufficiently fast parallel filters and under an assumption of strong differential observability of order $m$.

The article is structured as follows. Firstly, we present a general theorem that establishes the existence of a Lipschitz injective application $\dT$ which, when applied on the state trajectory is a specific solution to the nonlinear filter equation for sufficintly large parameter $k$. The next section provides a demonstration, presented as a series of propositions, with proofs provided in the Appendix. In the following section, we offer a preliminary illustration of these results within a robustness context. Lastly, the conclusion is provided.

\textit{Notation.} For a subset $\OO$ of $\RR^n$, we denote by $\OO+\delta$ the set
$\OO+\delta=\left\{x\in\RR^n|\hspace{0.1cm}\exists x_0\in\OO,|x-x_0|\leq\delta\right\}$.
When needed, to exhibit the dependency on initial conditions $x\in \RR^n$, we denote, when defined, $X(x,t)$ the (unique) solution to \eqref{eq:syst_f} at time $t$, initialized at $x$. 
By $L_f h$ we denote the Lie derivative of $h$ alongside the vector field $f$, i.e.,
$
L_f h(x)= \frac{\partial h}{\partial x}(x)f(x)
$, for all $x\in \RR^n$.
Given a map $g:\RR^n\to \RR^p$, we define $\Delta g : \RR^n \times \RR^n \to \RR^p$ by $\Delta g(x_a,x_b) = g(x_a)-g(x_b)$ for $(x_a,x_b)\in \RR^n\times \RR^n$, and we say that $g$ is ($k$-)Lipschitz injective on $S\subset \RR^n$, if there exists $k>0$ such that for all $(x_a,x_b)\in S\times S$, $|\Delta g(x_a,x_b)| \geq k|x_a-x_b|$.
Given a polynomial in powers of $\frac{1}{\lambda}$ in the form $p(\lambda)=\sum_{j=0}^\ell \frac{p_j}{\lambda^j}$ and given an integer $m$, the notation 
 $\leftrestrict p(\lambda) \rightrestrict_{\leq i}^m$  (resp. $\leftrestrict p(\lambda) \rightrestrict_{= i}^m$) represents the polynomial in powers of $\frac{1}{\lambda}$ obtained by keeping only the monomials of power  $\frac{1}{\lambda ^j}$ with $j\in \{0,\ldots, i\}$ (resp. $j=i$) in  $p(\lambda)^m$.
 
\section{Main result}

We assume (i) the system solutions of interest, initialized in some set $\X_0\subset \RR^n$, are defined on $\RR_{\geq 0}$ and remain in a compact set $\X\subset\RR^n$, and (ii) the system \eqref{eq:syst} is \textit{strongly differentially observable} of order $m$ on $\X$, as detailed next.
\begin{assumption}
\label{ass:setO}
    There exists a compact set $\X\subset\RR^n$ such that for all solution of \eqref{eq:syst} such that $x(0)\in\X_0$ then $x(t)\in\X$ for all $t\geq 0$.
\end{assumption}
\begin{assumption}
\label{ass:Hlinj}
 There exists an integer $m\geq 1$
 such that the map $\dH_m:\RR^n\rightarrow\RR^m$ defined by:
    \begin{equation}
        \dH_m: x\mapsto \begin{pmatrix}
h(x)\\ 
L_f h(x)\\ 
\vdots\\ 
L_f^{m-1} h(x)
\end{pmatrix}
    \end{equation}
    is $\underline k_H$-Lipschitz injective on $\X$ for some constant $\underline k_H>0$.
\end{assumption}

The latter assumption is guaranteed as soon as $\dH_m$ is injective and its jacobian is full-rank on $\X$, i.e., when $\dH_m$ is an injective immersion on $\X$.

In this paper, we show that the system \eqref{eq:syst} can then be transformed through a left-invertible change of coordinates into a  contracting filter \eqref{eq:zfilter_nonlin} with $\dsigma$  consisting of $m$ parallel contracting filters of the output in the form
\begin{equation}
\label{eq_sigma}
\dsigma(z,y) = \begin{pmatrix}
    \lambda_0\sigma(z_0,y)\\
    \vdots\\
    \lambda_{m-1} \sigma(z_{m-1},y)
\end{pmatrix}
\end{equation}
where $\lambda_0,\hdots,\lambda_{m-1}$ are positive distinct scalars, and $\sigma : \RR \times \RR \to \RR$  satisfying:
\begin{assumption}
\label{ass:contraction}
The map $\sigma : \RR \times \RR \to \RR$ is of class $C^{m+1}$ and verifies for all $(z,y)\in \RR\times \RR$,
\begin{subequations}
\label{eq:condsigma}
\begin{align}
\label{eq:condsigmay}
    0 < \gamma \leq &\left| \frac{\partial \sigma}{\partial y} (z,y)\right| \\
    \label{eq:condsigmaz}
    -\beta \leq  &\frac{\partial \sigma}{\partial z}(z,y)  \leq -\alpha < 0
\end{align}
\end{subequations}
for some constants $\alpha, \beta, \gamma > 0$.
\end{assumption}

It can be noticed that the map $\dsigma$ defined in \eqref{eq_sigma} then verifies
\begin{equation}
\frac{\partial\dsigma}{\partial \dz}(\dz,y) + \frac{\partial\dsigma}{\partial \dz}(\dz,y)^\top 
< - \mu\, \mathrm{I}_m\ , \forall(\dz,y) \in\RR^m\times\RR,
\end{equation}
where $\mu=-\alpha\min_{0\leq i\leq m-1} \lambda_i>0$
which
according to \cite[Theorem 1]{pavlov_uniformRegulation}, guarantees that  the filter \eqref{eq:zfilter_nonlin} is \textit{uniformly convergent} in the sense of \cite{pavlov_uniformRegulation}.
From this, we show the following result.

\begin{theorem}
	\label{theo_T}
	Under Assumptions \ref{ass:setO}, \ref{ass:Hlinj} and \ref{ass:contraction}, 
	for all $m$-uplet $(\lambda_0,\dots,\lambda_{m-1})$ of distinct positive scalars,
	there exists
	$k^*>0$ such that for all $k>k^*$,
	there exists a map
    $\dT:\RR^n\to\RR^m$
    that is Lipschitz injective on $\X$
    and such that,
    for all solution $t\mapsto x(t)$ of \eqref{eq:syst_f} with initial condition in $\X_0$,
	$t\mapsto \dT(x(t))$ is solution to \eqref{eq:zfilter_nonlin} with $t\mapsto y(t)$ given by \eqref{eq:syst_h} and $\dsigma$ defined in \eqref{eq_sigma}.
\end{theorem}

This result is proved in Section \ref{sec:proof}.
Under the conditions of the above theorem, we deduce from the injectivity of $\dT$ that there exists a continuous map $\Tinv:\RR^m\to\RR^n$ that is a left-inverse of $\dT$ on $\X$, namely 
$$
\Tinv(\dT(x)) = x \quad \forall x\in \X \ .
$$
For any such left-inverse,  any solution $(x, z)$ of \eqref{eq:syst}-\eqref{eq:zfilter_nonlin}
initialized in $\X_0\times\RR^m$,
with $\dsigma$ defined in \eqref{eq_sigma},
verifies
$\lim_{t\to+\infty}(\hat x(t) - x(t))=0$,
with $\hat x$ given by \eqref{eq:xhat} (see for instance  \cite[Theorem 1.1]{brivadis_further}).
Note that from the \textit{Lipschitz}-injectity of $\dT$, $\Tinv$ can even be picked globally Lipschitz, providing \textit{exponential} convergence of the estimation error as in \cite{andrieu_convergenceSpeed}.
%

\section{Proof}
\label{sec:proof}

\subsection{Some preliminaries}
For each $y$ in $\RR$, the function $z\mapsto \sigma(z,y)$ is decreasing and is a bijection from $\RR$ to $\RR$ from Assumption \ref{ass:contraction}.
Hence,
there exists a unique map $\psi:\RR\to\RR$ such that for all $y\in \RR$,
\begin{equation}
\label{eq:defpsi}
   \sigma(\psi(y),y) = 0. 
\end{equation}
Moreover, according to the implicit function theorem, $\psi\in C^{m+1}(\RR, \RR)$ since $\sigma\in C^{m+1}(\RR, \RR)$.

Also,
it can be observed that, from Assumption \ref{ass:setO}, the solutions to \eqref{eq:syst_f} initialized in $\X_0$ coincide in positive time with that of the modified dynamics
\begin{equation}
\label{eq:syst_modified}
    \dot{x} = \Breve{f}(x) = \chi(x)f(x),
\end{equation}
with $\chi$ a smooth function such that
\begin{equation}
    \chi(x) = 1 \hspace{0.2cm}\text{if}\hspace{0.2cm} x\in\X, \hspace{0.5cm}\chi(x) = 0 \hspace{0.2cm}\text{if}\hspace{0.2cm} x\notin\X+\delta_u \nonumber
\end{equation}
for some $\delta_u>0$. Besides, Assumptions \ref{ass:setO} and \ref{ass:Hlinj} still hold for \eqref{eq:syst_modified} with output map \eqref{eq:syst_h}. In the rest of the proof, we thus assume without loss of generality that $f=\Breve{f}$ and we consider that Assumption \ref{ass:setO}, whenever used,  ensures  (forward and backward) invariance of the compact set $\X+\delta_u$
as well as completeness and boundedness of all solutions.

\subsection{Construction of $\dT$}
\begin{proposition}
	\label{prop_TExist}
	Under Assumptions \ref{ass:setO} and \ref{ass:contraction}, 
	for all $m$-uplet $(\lambda_0,\dots,\lambda_{m-1})$ of distinct positive scalars, and for all $k>0$
	there exists
		$\dT:\RR^n\to\RR^m$
	 such that,
	for all solution $t\mapsto x(t)$ of \eqref{eq:syst_f} with initial condition in $\RR^n$,
	$t\mapsto \dT(x(t))$ is solution to \eqref{eq:zfilter_nonlin} with $t\mapsto y(t)$ given by \eqref{eq:syst_h} and $\dsigma$ defined in \eqref{eq_sigma}.
\end{proposition}
\begin{proof}
With Assumption \ref{ass:setO}, for all $x\in\RR^n$, the function
$$
s\mapsto y(s):=h(X(x,s))
$$
is defined, smooth, and bounded for all positive time.
According to \cite[Theorem 1]{pavlov_uniformRegulation}, Assumption \ref{ass:contraction} guarantees that  for all $x$ in $\RR^n$, the system
\begin{equation}\label{eq_scalarfilter}
\dot z = \lambda \sigma(z,y)
\end{equation}
with state $z\in\RR$ and $\lambda>0$
admits a unique bounded solution $t\mapsto \Zstar(t,y,\lambda)$ defined on $\RR$, which is uniformly globally asymptotically stable.
For all $x\in\RR^n$ and all $\lambda > 0$, we then define
\begin{equation}\label{eq_T}
    T(x,\lambda) = \Zstar(0,s\mapsto h(X(x,s)),\lambda),
\end{equation}
and for all $\lambda_0,\dots,\lambda_{m-1}>0$, and all $k> 0$,
\begin{equation}
\label{eq:def_dT}
    \dT(x)
    = \begin{pmatrix}
        T(x,k\lambda_0)\\
        T(x,k\lambda_1)\\
        \vdots\\
        T(x,k\lambda_{m-1})
    \end{pmatrix}.
\end{equation}

We then show that for any $x\in \RR^n$, $t\mapsto \dT(X(x,t))$ is solution to \eqref{eq:zfilter_nonlin}.
To do so, it is sufficient to show that for any $\lambda, k>0$,
$T(X(x,t),k\lambda) = \Zstar(t,s\mapsto h(X(x,s)),k\lambda)$ for all $t\in\RR$.
For all $t$ in $\RR$, we have
\begin{align}
    T(X(x,t),k\lambda) &= \Zstar(0,s\mapsto h(X(X(x,t),s)),k\lambda)\nonumber \\
    &= \Zstar(0,s\mapsto h(X(x,t+s)),k\lambda). \nonumber
\end{align}
It is thus enough to show that for each $t\in \RR$,
\begin{equation}
\label{eq:ZstarTimeOffset}
    \Zstar(0,s\mapsto h(X(x,t+s)),k\lambda)=\Zstar(t,s\mapsto h(X(x,s)),k\lambda).
\end{equation}
Let $t\in \RR$. 
For any bounded $\tau\mapsto y(\tau)$, the system
\begin{equation}
\label{modifiedsystem}
    \frac{\dd z}{\dd \tau}(\tau)=k\lambda\sigma(z(\tau),y(\tau)),
    \quad \forall\tau\in\RR,
\end{equation} 
is uniformly convergent in the sense of \cite{pavlov_uniformRegulation}. Therefore, it admits a unique bounded solution defined on $\RR$ that is $\tau\mapsto \Zstar(\tau,s\mapsto y(s),k\lambda)$. Consider the map $z_t:\tau\mapsto \Zstar(\tau+t,s\mapsto h(X(x,s)),k\lambda)$. It is bounded on $\RR$ and verifies, by definition of $\Zstar$,
\begin{multline*}
    \frac{\dd z_t}{\dd\tau}(\tau) =\\k\lambda\sigma(\Zstar(\tau+t,s\mapsto h(X(x,s)),k\lambda),h(X(x,t+\tau))). 
\end{multline*}
So by uniqueness of the bounded solution of \eqref{modifiedsystem}, 
$$
z_t(\tau) = \Zstar(\tau,s\mapsto h(X(x,t+s)),k\lambda) .
$$
Taking $\tau=0$ yields \eqref{eq:ZstarTimeOffset} and concludes the proof.
\end{proof}


\subsection{Construction of an approximation of $\dT$ which is Lipschitz injective}
To prove Theorem \ref{theo_T}, we need to show that the mapping $\dT$ obtained from Proposition \ref{prop_TExist} is Lipschitz injective on $\X$ for sufficiently large $k$.
This is made difficult by the fact that we have no idea of the regularity of the function $\dT$.
The idea of the proof is to use an approximation of the components of $\dT$ which is $C^1$ (i.e. $T(\cdot, k\lambda_i)$), and more precisely a development in powers of $\frac{1}{k}$.
The approximation we consider is obtained from a $C^1$  mapping $\dphi:\RR^n\to\RR^m$, with
$\dphi=(\phi_0,\dots,\phi_{m-1})$ where $\phi_\ell$ are of class $C^{m-\ell+1}$ and are
 defined recursively for $x$ in $\RR^n$ as
\begin{equation}
\label{eq:def_phi01}
    \phi_0(x) = \psi(h(x))
    ,
\end{equation}
and for $\ell\in \{1,\dots, m-1\}$:
\begin{multline}\label{eq_phii}
    \phi_\ell(x) = \\
    \kappa(h(x))\biggr[L_f \phi_{\ell-1}(x)
     - \sum_{j=1}^{\ell} \frac{1}{j!}\frac{\partial^j\sigma}{\partial z^j}(\psi(h(x)),h(x))\\\times\sum_{\substack{\ell_1+\dots+\ell_j = \ell
     \\
     1\leq \ell_1,\dots,\ell_j\leq \ell-1
     }}
    \phi_{\ell_1}(x)\cdots\phi_{\ell_j}(x)
    \biggr],
\end{multline}
where
\begin{equation}\label{tmetrm}
    \kappa(y) = \left(\cfrac{\partial\sigma}{\partial z}(\psi(y),y)\right)^{-1}.
\end{equation}
Notice that the definition of $\phi_\ell(x)$ for $\ell\geq 1$ involves $\phi_i(x)$ with $i \in \{1,\cdots,\ell-1\}$ only and is independent from $\lambda$.
The approximation of $T(x,\lambda)$ is then defined as
\begin{equation}
\label{eq:defTa}
    T_a(x,\lambda)=\sum_{\ell=0}^{m-1} \frac{\phi_\ell(x)}{\lambda^\ell}.
\end{equation}

From there, given $\dlambda=(\lambda_0,\dots,\lambda_{m-1})$ and given $k>1$, we thus get from \eqref{eq:def_dT} the following approximation of $\dT$
\begin{equation}
\dTa(x) = \V K^{-1} \dphi(x) ,
\end{equation}
where $K=\text{diag}\left(1,\dots,k^{m-1}\right)$ and
$\V$ is the Vandermonde matrix
\begin{equation}
    \mathcal{V} = \begin{pmatrix}
1 & \lambda_0^{-1} & \dots & \lambda_0^{-(m-1)} \\ 
\vdots & \vdots & \ddots & \vdots  \\ 
1 & \lambda_{m-1}^{-1} & \dots &   \lambda_{m-1}^{-(m-1)}
\end{pmatrix}.
\end{equation}
Clearly, injectivity of $\dTa$ may be deduced from the injectivity of $\dphi$ as soon as the $\lambda_i$'s are all distinct. Actually, in the following proposition, it is shown that with Assumption \ref{ass:Hlinj}, injectivity of $\dphi$ is ensured on $\X$.
\begin{proposition}\label{Prop_LipInjPhi}
    Under Assumption \ref{ass:Hlinj}, 
    $\dphi$ is Lipschitz injective in $\X$.
    In other word, there exists a positive real number $\underline k_{\dphi}$ such that for all $(x_a,x_b)$ in $\X^2$
    \begin{equation}
        \left|\dphi(x_a)-\dphi(x_b)\right|\geq \underline k_{\dphi} \left|x_a-x_b\right|.
    \end{equation}
\end{proposition}
\if\Longversion1
The proof of Proposition \ref{Prop_LipInjPhi} can be found in Section
\ref{sec_Prop_LipInjPhi}.
\else 
The proof of Proposition \ref{Prop_LipInjPhi} can be found in \cite[Section
A.5]{KKLNonLinearLongVersion}.
\fi
A direct consequence  is that  for all  $m$-uplet $(\lambda_0,\dots,\lambda_{m-1})$ of positive distinct scalars, all $k>1$,  all $(x_a,x_b)$ in $\X^2$ the following inequality is satisfied
\begin{multline}\label{eq_TaInj}
|\dTa(x_a)-\dTa(x_b)| 
\geq \frac{ \underline{k}_{\dphi}}{k^{m-1} \|\V^{-1}\|}|x_a-x_b|,
\end{multline}
since $\|K\|\leq k^{m-1}$,
which establishes that $\dTa$ is Lipschitz injective in $\X$.

\subsection{$T_a$ is an approximation of order $\lambda^{-m}$ of $T$}
The reason for stating that $T_a$ is an approximation of $T$ will be shown in this section. Indeed, we demonstrate that the difference between these two functions is of the order $\lambda^{-m}$. In Proposition \ref{prop_TExist}, we have shown that for any $\lambda_0,\ldots,\lambda_{m-1}>0$, any $k>0$ and any $x\in \RR^n$, $t\mapsto \dT(X(x,t))$ is solution to \eqref{eq:zfilter_nonlin}. From the definitions of $\dsigma$ and $\dT$ in \eqref{eq_sigma} and \eqref{eq:def_dT}, we deduce that $T$ is supposed to verify, for all $x\in \RR^n$ and all $\lambda>0$,
\begin{equation}
\label{eq:PDE}
 L_f T(x,\lambda) = \lambda \sigma(T(x,\lambda),h(x)).
\end{equation}
In order to establish that $T_a$ approximates $T$, we first study the error in the partial differential equation \eqref{eq:PDE} induced by this approximation. 


\begin{proposition}\label{Prop_Omega}
Let $\omega:\RR^n\times\RR_{>0}\to\RR$ be defined as
\begin{equation} \nonumber
    \omega(x,\lambda) = \frac{L_f T_a(x, \lambda) - \lambda \sigma(T_a(x,\lambda), h(x))}{\lambda}
\end{equation}
with $T_a$ defined in \eqref{eq:defTa}.
Under Assumptions \ref{ass:setO} and \ref{ass:contraction}
there exist two positive real numbers $\bar{\omega}$ and $k_\omega$ such that
\begin{equation}
\label{eq:omegabound}
\forall x\in\X+\delta_u,\forall \lambda>1 ,\hspace{0.5cm}
    |\omega (x, \lambda)| \leq \frac{\bar{\omega}}{\lambda^m},
\end{equation}
and for all $(x_a,x_b)\in(\X+\delta_u)^2,\forall \lambda> 1$,
\begin{equation}
\label{eq_Deltaomegabound}
    |\omega (x_a, \lambda)-\omega(x_b,\lambda)| \leq \frac{k_\omega}{\lambda^m}|x_a-x_b|.
\end{equation}
\end{proposition}
\if\Longversion1
The proof of Proposition \ref{Prop_Omega} can be found in Section
\ref{sec_Prop_Omega}.
\else 
The proof of Proposition \ref{Prop_Omega} can be found in \cite[Section
A.1]{KKLNonLinearLongVersion}.
\fi
We introduce $R$  the difference between $T$ and its approximation $T_a$:
 \begin{equation}
R(x,\lambda)=T(x,\lambda)-T_a(x,\lambda).
 \end{equation}
Employing the bounds obtained in Proposition \ref{Prop_Omega}, the following two propositions establish that $R$ is bounded and is Lipschitz \textit{with order} $\frac{1}{\lambda^m}$:
\begin{proposition}\label{Prop_RBounded}
Under Assumption \ref{ass:setO} and \ref{ass:contraction},
\begin{equation}
\label{Rmbounded}
\forall x\in\X+\delta_u,\forall \lambda>1 ,\hspace{0.5cm}    |R(x,\lambda)|\leq \frac{\bar{\omega}}{\alpha \lambda^m}.
\end{equation}
\end{proposition}

\begin{proposition}\label{Prop_RLipsch}
Under Assumption \ref{ass:setO} and \ref{ass:contraction}, there exist $\lambda^*>1$ and a positive real number $k_R$ such that for all $(x_a,x_b)$ in $(\X+\delta_u)^2$ for all $\lambda>\lambda^*$,
\begin{equation}\label{eq_BoundDeltaR}
     |R(x_a,\lambda)-R(x_b,\lambda)|\leq \frac{k_R}{\lambda^m}|x_a-x_b|.
\end{equation}
\end{proposition}
\if\Longversion1
The proof of Proposition \ref{Prop_RBounded} and \ref{Prop_RLipsch} can be found in 
Section \ref{sec_Prop_RBounded} and Section \ref{sec_Prop_RLipsch}.
\else 
The proof of Proposition \ref{Prop_RBounded} can be found in \cite[Section
A.2-A.4]{KKLNonLinearLongVersion}.
\fi

\subsection{Proof of Theorem \ref{theo_T}} 

For $k>1$ and for a given $m$-uplet $\dlambda$ of distinct positive real numbers, let $\dT$ be given by Proposition \ref{prop_TExist}. 
For all solution $t\mapsto x(t)$ of \eqref{eq:syst_f} with initial condition in $\X_0$,
	$t\mapsto \dT(x(t))$ is solution to \eqref{eq:zfilter_nonlin} with $t\mapsto y(t)$ given by \eqref{eq:syst_h} and $\dsigma$ defined in \eqref{eq_sigma}.
 With $\dTa$ given previously, we can introduce the mapping $\dR$
as
\begin{equation}
    \dT(x) = \dTa(x) + \dR(x),
\end{equation}
with
$$
\dR(x) = (R(x,k\lambda_0),\dots,R(x,k\lambda_{m-1})).
$$
With Proposition \ref{Prop_LipInjPhi}, equation \eqref{eq_TaInj} and Proposition \eqref{eq_BoundDeltaR}, we conclude that for all $k$ such that $k\lambda_i>\lambda^*$ for all $i=0,\dots,m-1$,
\begin{multline*}
    |\Delta \dT(x_a,x_b)|
    \geq \frac{1}{k^{m-1}}\left(\frac{ \underline{k}_{\dphi}}{\|\V^{-1}\|} - \frac{1}{k} \sum_{i=0}^{m-1} \frac{k_R}{\lambda_i^m} \right)|x_a-x_b|.
\end{multline*}
This proves that there exists  $k^*>0$ such that, for any $k\geq k^*$, $x\mapsto \dT(x)$ given in \eqref{eq:def_dT} is Lipschitz injective on $\X$ which concludes the proof of Theorem \ref{theo_T}.

\section{Illustration}

In this part, we show how a KKL observer with \textit{nonlinear} filter dynamics may allow to obtain simultaneously fast convergence and robustness to measurement noise. For that, we compare its performance to two KKL observers with \textit{linear} filter dynamics, one \textit{fast} and the other \textit{slow}
\begin{equation}
\label{eq:linfilt_simu}
    \dot z = \lambda a_\textrm{fast}(z-y),
    \quad
     \dot z = \lambda a_\textrm{slow}(z-y),
\end{equation}
where $\lambda>0$, and
$a_\textrm{fast}<a_\textrm{slow}<0$.
When choosing a family $(\lambda_i)_{1\leq i\leq m}$, the corresponding observers are given by \eqref{eq:zfilter_lin} with $A_\textrm{fast} = a_\textrm{fast}\diag(\lambda_i)_{1\leq i\leq m}$ and $B_\textrm{fast} = -a_\textrm{fast}(1,\dots,1)^\top$ (resp. $A_\textrm{slow} = a_\textrm{slow}\diag(\lambda_i)_{1\leq i\leq m}$ and $B_\textrm{slow} = -a_\textrm{slow}(1,\dots,1)^\top$).
It is well-known that when tuning the parameters of linear filters, a compromise has to be found between speed of convergence and robustness to noise.
More precisely, modulo the nonlinear left-inversion of $\dT$,  KKL observers built with fast (resp. slow) linear filters usually exhibit fast (resp. slow) convergence properties but poor (resp. good) robustness to measurement noise.

On the other hand, building a KKL observer based on nonlinear filters allows to consider dynamics of the form
\begin{equation}
\label{eq:nonlinfilt_simu}
    \dot{z} = \lambda(a_\textrm{fast}(z-y)+(a_\textrm{slow}-a_\textrm{fast})\textrm{tanh}(z-y))
\end{equation}
which can be checked to verify Assumption \ref{ass:contraction}. 
When choosing a family $(\lambda_i)_{1\leq i\leq m}$, the corresponding observer is given by \eqref{eq:zfilter_nonlin} for $k=1$ and the corresponding contraction $\dsigma$ as in \eqref{eq_sigma}.
The motivation is the following:
On the one hand,
for large values of $|z-y|$,
the nonlinear dynamics behave as the \textit{fast} linear dynamics. Therefore, a fast convergence is expected.
On the other hand, for small values of $|z-y|$,
the nonlinear dynamics behave as the \textit{slow} linear dynamics. Therefore, the same robustness with respect to small perturbations of the measurement $y$ is expected.
In other words, the nonlinear KKL observers is expected to take the best of both worlds.

For simulations, we consider a nonlinear Duffing oscillator
\begin{equation}
\label{eq:syst_simu}
    \left\{
    \begin{array}{ll}
        \dot x_1 = x_2 \\
        \dot x_2 = -0.2 x_1-x_1^3 
    \end{array}
\right. 
\quad , \quad y = x_1,
\end{equation}
which verifies Assumptions \ref{ass:setO} and \ref{ass:Hlinj} with $\X_0 = [-2,2]^2$ and $m=2$.
We pick $\lambda_1 = -2$, $\lambda_2 = -4$, $\lambda_3 = -6$, $a_\textrm{fast} = 5$ and $a_\textrm{slow} = 0.5$. For each KKL observer, we create a dataset of points $(x,\dz)$ approximating $(x,\dT(x))$, with the map $\dT$ such that the image of solutions to \eqref{eq:syst_simu} by $\dT$ is solution to the corresponding observer dynamics \eqref{eq:zfilter_lin} or \eqref{eq:zfilter_nonlin}. This is done by simulating the interconnection of \eqref{eq:syst_simu} with \eqref{eq:zfilter_lin} or \eqref{eq:zfilter_nonlin} from a grid of $200\times 200$ initial conditions in $\X_0\times \{0\}$, and storing the obtained pairs $(x,\dz)$ after a time $t_l:=20/\min\{ |\lambda_i|\}$, needed for the filters to ``forget'' their initial condition. Values of $\dT(x)$ for a given $x$ and of $\Tinv(\dz)$ for a given $\dz$ can then be obtained by fetching the closest point in the dataset.

We then propose two simulation scenarios: 1) we initialize all observers with the same initial error randomly picked so that $|\dz(0) - \dT(x(0))| = 100$ and we do not add measurement noise, and 2) we initialize all observers to their correct initial condition $\dT(x(0))$ and we add a sinusoidal measurement noise ($\nu(t) = 0.1\sin(10t)$). The first (resp. second) scenario aims at comparing the convergence times (resp.  the impact of measurement noise). The numerical results for a particular choice of $x(0)$ are provided in Figure \ref{fig_errors_simu}. As expected, the KKL observer with nonlinear dynamics converges as fast as the one with fast linear dynamics and seems almost as robust to noise than the one with slow linear dynamics.  Minimum, maximum and mean convergence time and gain with respect to noise for 100 random $x(0)$ are given in Table \ref{tab}. The convergence time is computed in Scenario 1 as the first time after which the error remains below tolerance thresholds, determined by the precision of the approximation of $\dT$ and $\Tinv$. The robustness to measurement noise is quantified in Scenario 2 by dividing the 2-norm of the steady state error by the amplitude of the noise.
We refer the reader to \cite{github-KKL-NL} to experiment the simulations.
Note that a high-gain observer with variable gains or nonlinear correction terms could also have been implemented for this particular system as in \cite{sanfelicepraly2011,boizotadaptive,prasovkhalil2013}.
Here, the tuning of $a_\textrm{fast}$ and $a_\textrm{slow}$ in \eqref{eq:nonlinfilt_simu} is straightforward, although the tuning of the $\lambda_i$'s is less, as well as their impact on the performance in the $x$-coordinates, since the nonlinear map $\Tinv$ used to obtain $\hat{x}$ from $\hat{z}$ depends in a complex way on those parameters.

\begin{figure}[htbp]
	\begin{subfigure}[t]{0.45\columnwidth}
		\centering
		\includegraphics[width=\columnwidth]{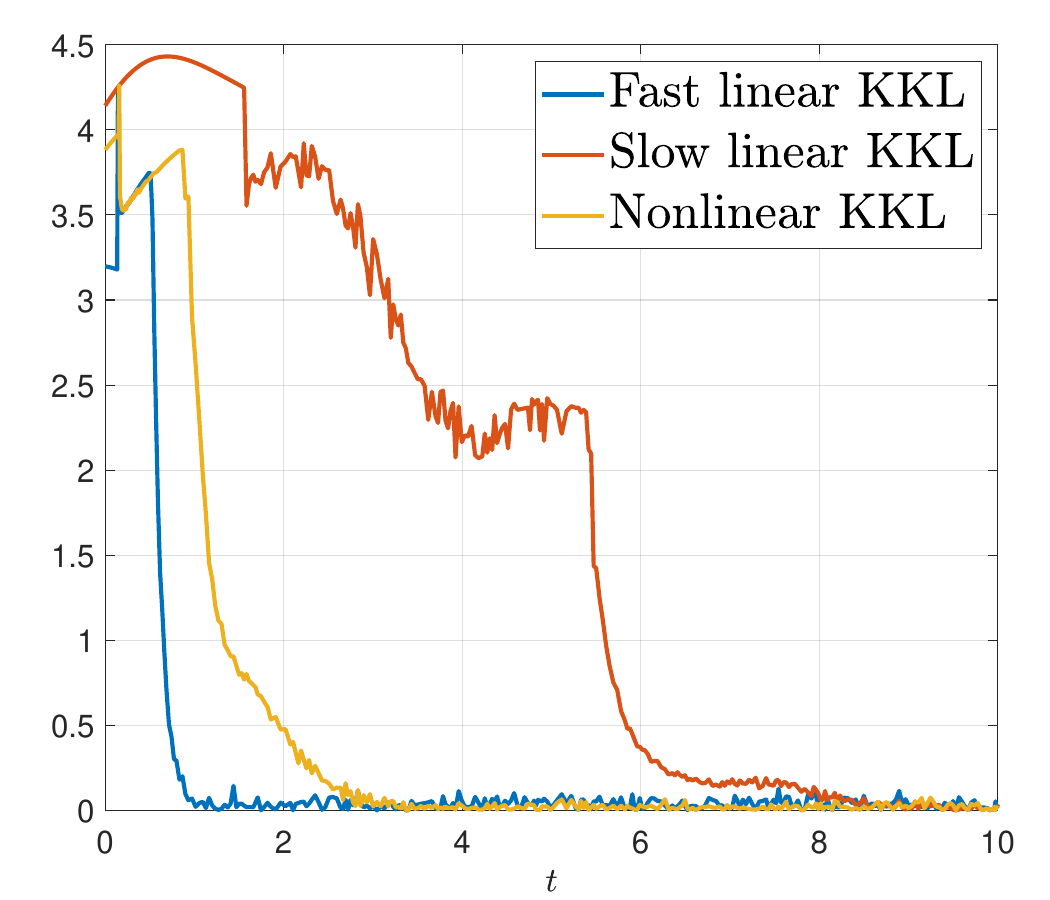}  
		\caption{Scenario 1: $|\hat{x}(t)-x(t)|$}
		\label{fig_error_x_conv}
	\end{subfigure}
\hspace{2em}
	\begin{subfigure}[t]{0.45\columnwidth}
		\centering
		\includegraphics[width=\columnwidth]{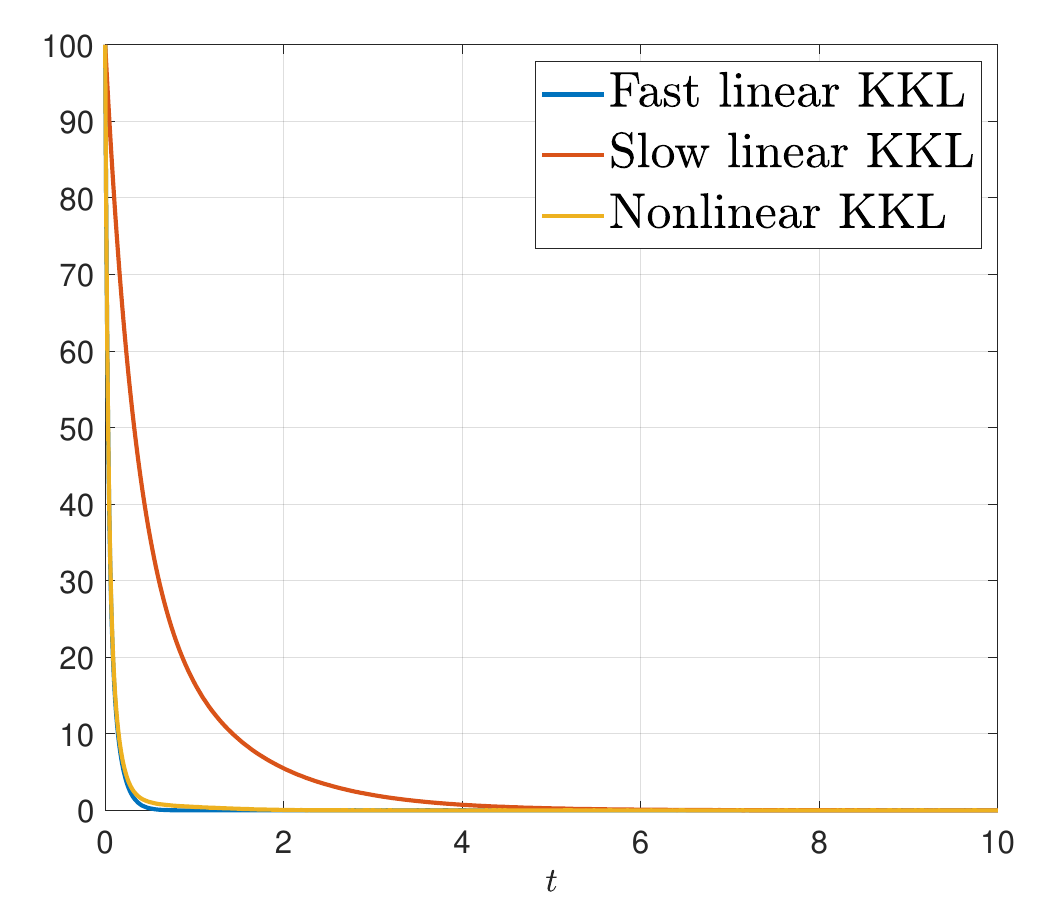}
		\caption{Scenario 1: $|\dz(t)-\dT(x(t))|$}
		\label{fig_error_z_conv}
	\end{subfigure}
 \begin{subfigure}[t]{0.45\columnwidth}
		\centering
		\includegraphics[width=\columnwidth]{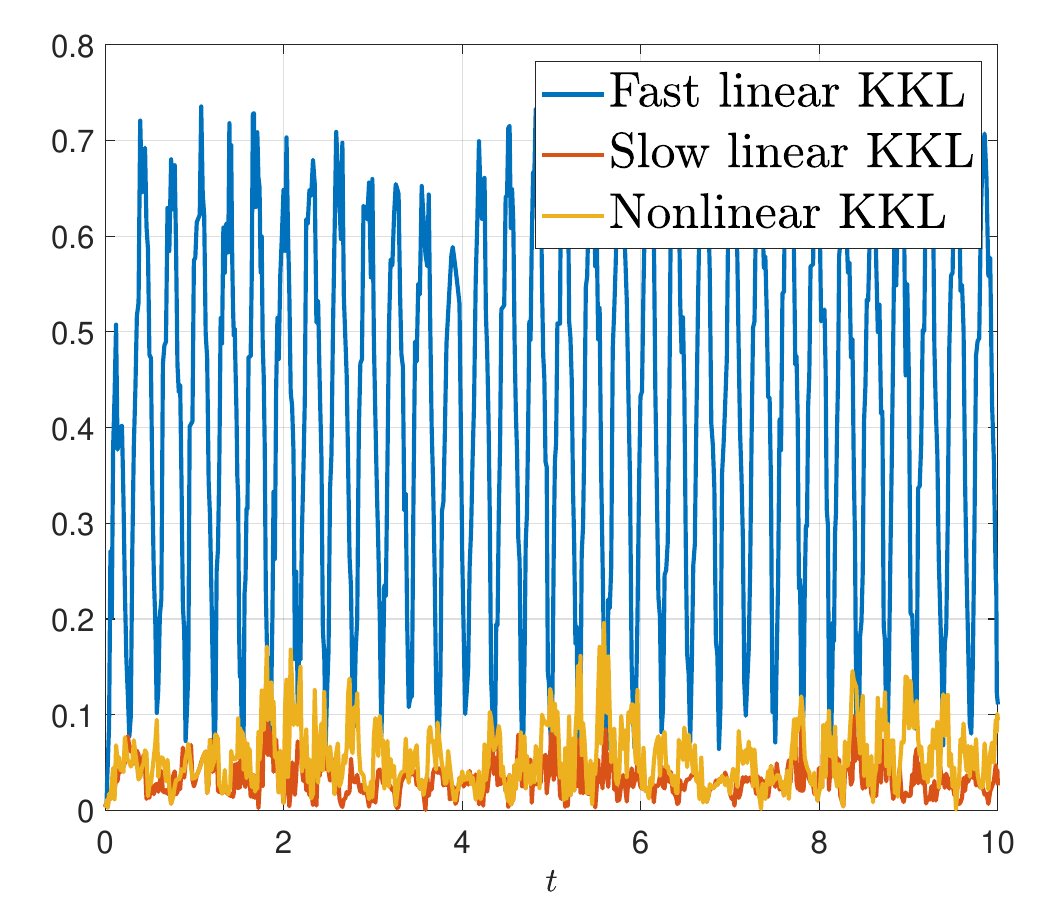}  
		\caption{Scenario 2: $|\hat{x}(t)-x(t)|$}
		\label{fig_error_x_noise}
	\end{subfigure}
\hspace{2em}
	\begin{subfigure}[t]{0.45\columnwidth}
		\centering
		\includegraphics[width=\columnwidth]{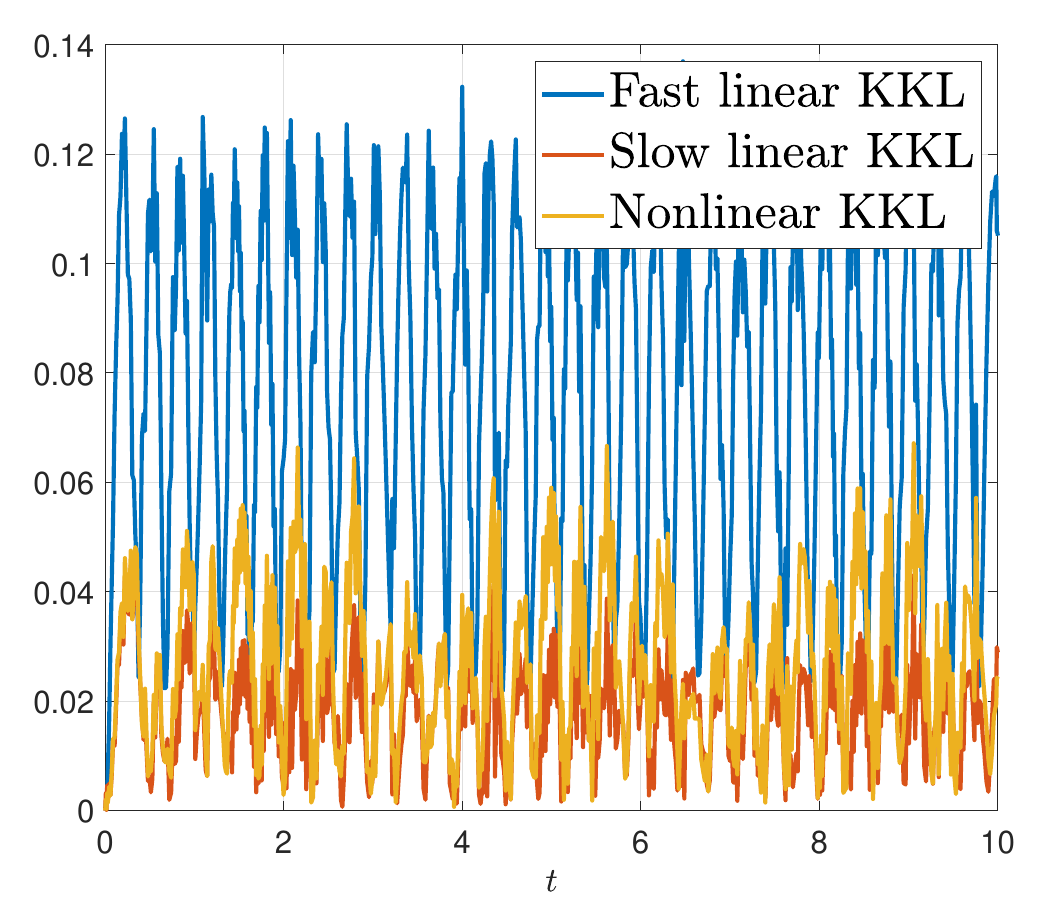}
		\caption{Scenario 2: $|\dz(t)-\dT(x(t))|$}
		\label{fig_error_z_noise}
	\end{subfigure}
	\caption{Estimation error in $x$ and $z$ coordinates from two KKL observers with linear dynamics \eqref{eq:zfilter_lin} with slow and fast pairs $(A_\textrm{slow},B_\textrm{slow})$ and $(A_\textrm{fast},B_\textrm{fast})$ defined from \eqref{eq:linfilt_simu}, and by a KKL observer with nonlinear dynamics \eqref{eq:zfilter_nonlin} defined from \eqref{eq:nonlinfilt_simu}. Scenario 1: initial error but no noise; Scenario 2 : no initial error but noise.}
	\label{fig_errors_simu}
\end{figure}


\begin{table}[]
    \centering
    \begin{tabular}{@{}c@{}|c|c|c|c}
    \multicolumn{2}{c|}{}   & Fast linear & Slow linear & Nonlinear \\
      \hline
     \multirow{3}{*}{
     $\begin{matrix}
         \text{Convergence}\\
         \text{time}
     \end{matrix}$
     } & Min & 0.78  & 6.29 & 1.52  \\
                                          & Max & 0.87 & 7.72 &  3.15\\
                                          & Mean & 0.83  &  6.79 & 2.27  \\
     \hline
     \multirow{3}{*}{ $\begin{matrix}
         \text{Gain}\\
         \text{w.r.t. noise}
     \end{matrix}$} & Min & 7.23 & 0.94 & 1.56 \\
                                    & Max & 8.07 & 1.53 & 2.37  \\
                                    & Mean & 7.57  & 1.15 &  1.95 \\
     \hline
    \end{tabular}
    \caption{Comparison of KKL observers with slow/fast linear dynamics and KKL observer with nonlinear dynamics in terms of convergence time and gain with respect to noise of  $|\hat{x}-x|$. }
    \label{tab}
\end{table}

\section{Conclusion}
In this article, we have presented a KKL-type observer with non-linear dynamics. In the scenario where the system is differentially observable of order $m$, we have demonstrated the existence of such an estimation algorithm when the observer's dynamics are structured as $m$ non-linear filters operating in parallel, provided that their dynamics are sufficiently fast.
Through a simplified illustration, we have highlighted a potential application of this technique to obtain a better qualitative behavior of the estimate. In a more general context, this paper opens the road to demonstrating the existence of such an observer for a broader class of contractions.

\bibliography{biblio}

\if\Longversion1

\appendix
\section{Proofs}
\textit{This part of the paper is not in the version which has been published in \cite{KKLNonLinearMICNON}.}

In this section, we give the technical proofs employed to get the main result.
Note that to simplify the presentation, we use the following notation:
$\sigma^{(j)}(x)=\frac{\partial^j \sigma}{\partial z^j}(\psi(h(x)),h(x))$.

\subsection{Proof of Proposition \ref{Prop_Omega}}
\label{sec_Prop_Omega}

First of all, we have the following technical lemma whose proof can be found in Appendix \ref{sec_ProofLemmaLfTa}.
\begin{lemma}\label{lemm_LfTa}
For each $x\in \RR^n$ and $\lambda>0$,
\begin{multline}
        \label{lfTa}
        L_f T_a(x,\lambda) = \lambda \sum_{j=1}^{m-1} \frac{\sigma^{(j)}(x)}{j!} 
\biggr[\sum_{\ell=1}^{m-1}\frac{\phi_\ell(x)}{\lambda^\ell}\biggr]\biggr|^{j}_{\leq m-1} \\
         + \frac{1}{\lambda^{m-1}}L_f\phi_{m-1}(x).
    \end{multline}
\end{lemma}
Hence, with \eqref{tmetrm}, \eqref{eq:defTa}, \eqref{eq:def_phi01} and \eqref{lfTa}, it yields
\begin{equation}\nonumber
\begin{split}
    \omega(x,\lambda) = \sum_{i=1}^{m-1} \frac{\sigma^{(i)}(x)}{i!}\biggr[ \sum_{\ell=1}^{m-1} \frac{\phi_\ell(x)}{\lambda^\ell} \biggr]\biggr{|}^i_{\leq m-1} \\ -  \sigma\left(\psi(h(x)) + \sum_{\ell=1}^{m-1} \frac{\phi_\ell(x)}{\lambda^\ell}, h(x)\right) + \frac{1}{\lambda^m}L_f \phi_{m-1}(x)
    \end{split}
\end{equation}
As $\sigma$ is $C^{m+1}$, using \eqref{eq:defpsi} and a Taylor expansion  of $\sigma$ around $(\psi(h(x)),h(x))$,
$$
    \omega(x,\lambda) = \frac{\omega_{0}(x,\lambda)}{\lambda^m} - G(x,\lambda) + \frac{1}{\lambda^m}L_f \phi_{m-1}(x)
$$
where
\begin{multline*}
\omega_{0}(x,\lambda) = \frac{1}{(m-1)!}\int_0^1 \biggr[(1-s)^{m-1}\frac{\partial^m \sigma}{\partial z^m}\biggr(\psi(h(x))  \\+ s \sum_{\ell=1}^{m-1}\frac{\phi_\ell(x)}{\lambda^\ell}  ,h(x)\biggr)\biggr]\dd s \biggr[ \sum_{\ell=1}^{m-1}\frac{\phi_\ell(x)}{\lambda^{\ell-1}} \biggr]^m
\end{multline*}
$$
G(x,\lambda) = \sum_{i=1}^{m-1} \frac{\sigma^{(i)}(x)}{i!}
    \biggr( \biggr[ \sum_{\ell=1}^{m-1}\frac{\phi_\ell(x)}{\lambda^\ell} \biggr]^i- \biggr[ \sum_{\ell=1}^{m-1}\frac{\phi_\ell(x)}{\lambda^\ell} \biggr] \biggr|^i_{\leq m-1}\biggr)
$$
Note that $G$ equals $0$ if $i=1$ and equals  
\begin{equation}\nonumber
    \frac{1}{\lambda^m} \sum_{j=0}^{i(m-1)} \frac{1}{\lambda^j} \tilde{\phi}_{ij}(x) 
\end{equation}
if $i>1$, with $\tilde{\phi}_{ij}$ being a certain polynomial of the $\phi_\ell$ for $\ell=1,\dots,m-1$, therefore this gives
\begin{multline}
\label{eq_omegam}
    \omega(x,\lambda) = 
    \frac{\omega_{0}(x,\lambda)+L_f \phi_{m-1}(x)}{\lambda^m} 
    \\+ \frac{1}{\lambda^m}\sum_{i=2}^{m-1} \frac{\sigma^{(i)}(x)}{i!} \sum_{j=0}^{i(m-1)} \frac{\tilde{\phi}_{ij}(x)}{\lambda^j} .  
    \end{multline}
Therefore, using triangular inequality, for $x\in\X+\delta_u$ and $\lambda>0$,
\begin{multline}
|\lambda^m||\omega(x,\lambda)|\leq \left|\omega_{0}(x,\lambda)
    \right| + \left|L_f \phi_{m-1}(x)\right|
    \\+\sum_{i=2}^{m-1} \frac{1}{i!} \biggr| \sigma^{(i)}(x)\biggr| \sum_{j=0}^{i(m-1)} \frac{1}{\lambda^j} |\tilde{\phi}_{ij}(x)|.
\end{multline}
As $\sigma$ and $\psi$ are $C^{m+1}$, $\sigma^{(i)}$ is continuous, and therefore bounded on the compact set $\X+\delta_u$ for all $i$ in $\{1,\dots,m\}$. Moreover, $\tilde{\phi}_{kj}$ is a polynomial of the continuous functions $\phi_\ell$, and therefore is bounded on the compact set $\X+\delta_u$. As $\phi_{m-1}$ is $C^2$, $L_f \phi_{m-1}$ is also bounded on $\X+\delta_u$.
Moreover, there exists $\delta_\psi>0$ such that $\left|\sum_{\ell=1}^{m-1}\frac{\phi_\ell(x)}{\lambda^\ell}\right|\leq \delta_\psi$ for all $x\in \X+\delta_u$ and all $\lambda>1$. By continuity of $\frac{\partial^m\sigma}{\partial z^m}$, it is bounded on $(\psi(h(\X+\delta_u))+\delta_\psi)\times h(\X+\delta_u)$, uniformly in $\lambda>1$. We can conclude that there exists $\bar{\omega}>0$, independent from $\lambda$, such that \eqref{eq:omegabound} holds.
Note also that $\omega$ is $C^1$ and so are all the previously listed maps. So similarly, it yields,
\begin{multline}
\lambda^m\left|\frac{\partial \omega}{\partial x}(x,\lambda)\right|\leq \left|\frac{\partial\omega_{0}}{\partial x}(x,\lambda)
    \right| + \left|\frac{\partial L_f \phi_{m-1}}{\partial x}(x)\right|
    \\+\sum_{i=2}^{m-1} \frac{1}{i!} \biggr| \frac{\partial \sigma^{(i)}}{\partial x}(x)\biggr| \sum_{j=0}^{i(m-1)} \frac{1}{\lambda^j} \left|\frac{\partial \tilde{\phi}_{ij}}{\partial x}(x)\right|.
\end{multline}
Hence, there exists a positive real number $k_\omega'$, such that for all $\lambda>1$ and $x\in \X+\delta_u$, it yields
$$
\lambda^m\left|\frac{\partial \omega}{\partial x}(x,\lambda)\right|\leq k_\omega'.
$$
Inequality \eqref{eq_Deltaomegabound} is obtained for some $k_\omega>k_\omega'$.

\subsection{Proof of Proposition \ref{Prop_RBounded}}
\label{sec_Prop_RBounded}

The evaluation $R(X(x,t),\lambda)$ of $R$ along any solution initiated at $x\in\X+\delta_u$ satisfies
\begin{multline*}
\frac{d}{dt}R(X(x,t),\lambda)=\\\lambda[\sigma(T_a(X(x,t),\lambda) +R(X(x,t),\lambda), h(X(x,t)))\\ -\sigma(T_a(X(x,t),\lambda),h(X(x,t)))] - \lambda \omega(X(x,t),\lambda)
\end{multline*}
With the condition \eqref{eq:condsigmaz} it yields for all $x\in\X + \delta_u  $  
\begin{equation}\nonumber
\begin{split}
    R(x,\lambda)\big[\sigma(T(x,\lambda)+R(x,\lambda),h(x))-\sigma(T(x,\lambda),h(x))\big]\\
    \leq -\alpha R(x,\lambda)^2.
    \end{split}
\end{equation}
Since moreover, $\X+ \delta_u$ is negatively invariant along the flow and with Young's inequality, for all $\varepsilon>0$ it gives
\begin{align}
\begin{split} \nonumber
    \frac{d}{dt}R(X(x,t),\lambda)^2 &\leq -2\lambda\alpha R(X(x,t),\lambda)^2 \\
    &\qquad \hspace{-0.5cm} + \frac{\lambda}{\varepsilon}\omega(X(x,t),\lambda)^2 + \lambda\varepsilon R(X(x,t),\lambda)^2
    \end{split}
    \\ \nonumber
    & \hspace{-0.1cm} =-\lambda\alpha R(X(x,t),\lambda)^2 + \frac{\lambda}{\alpha}\omega(X(x,t),\lambda)^2
\end{align}
where we chose $\varepsilon=\alpha$ for the last equality to hold.
By multiplying each side by $e^{\lambda\alpha t}$ and rearranging the terms, we obtain
\begin{equation}\label{eq:rearrange}
    \frac{d}{dt}(R(X(x,t), \lambda)^2 e^{\lambda \alpha t}) \leq \frac{\lambda}{\alpha}\omega(X(x,t),\lambda)^2 e^{\lambda \alpha t}.
\end{equation}
Let $\delta > 0$. Integrating between $-\delta$ and $0$ and using \eqref{eq:omegabound}, we get for any $x\in\X+\delta_u$ and any $\lambda>1$,
\begin{align}
\begin{split} \nonumber
    R(x,\lambda)^2 & \leq R(X(x,-\delta), \lambda) ^2
     e^{-\lambda \alpha \delta} \\
    &\qquad + \frac{\lambda}{\alpha} \int_{-\delta}^0 \omega (X(x,s), \lambda)^2 e^{\lambda \alpha s}\dd s 
    \end{split}
    \\ \nonumber
    & \leq R(X(x,-\delta), \lambda)^2
     e^{-\lambda \alpha \delta} + \frac{\lambda}{\alpha} \frac{\bar{\omega}^2}{\lambda^{2m}} \int_{-\delta}^0 e^{\lambda \alpha s}\dd s 
    \\ \nonumber
    & = R(X(x,-\delta), \lambda)  ^2
      e^{-\lambda \alpha \delta} + \frac{\lambda \bar{\omega}^2(1 - e^{-\lambda \alpha \delta})}{\alpha^2 \lambda^{2m+1}}
\end{align}
$\delta\mapsto X(x,-\delta)$ is bounded, so with the definition of $T$ in \eqref{eq_T}, $\delta\mapsto T(X(x,-\delta),\lambda)$ is bounded, and, by continuity of $\phi_\ell$ for each $\ell\in \{0,\ldots,m-1\}$, $\delta\mapsto R(X(x,-\delta),\lambda)$ is also bounded.
Letting $\delta$ go to $+\infty$, it yields 
\begin{equation} \nonumber
    \forall\lambda>1,\forall x\in\X+\delta_u,\hspace{0.3cm}R(x,\lambda)^2 \leq \frac{\bar{\omega}^2}{\alpha^2 \lambda^{2m}}.
\end{equation}
We finally get \eqref{Rmbounded}.

\subsection{Proof of Lemma \ref{lemm_LfTa}}
\label{sec_ProofLemmaLfTa}
Define $T_{a,i}(x,\lambda)=\sum_{j=0}^{i-1} \frac{\phi_j(x)}{\lambda ^j}$,
so that $T_{a,m} = T_a$.
The proof of Lemma \ref{lemm_LfTa} is obtained by recursion. 
We are going to prove the following property for 
$i\geq 1$:

\noindent$\underline{\mathcal{P}(i)}$: For all $x\in \RR^n$ and $\lambda>0$, $T_{a,i}$ verifies
    \begin{multline}
     \label{lftk}
        L_f T_{a,i}(x,\lambda) = \lambda \sum_{j=1}^{i-1} \frac{\sigma^{(j)}(x)}{j!} \biggr[\sum_{\ell=1}^{i-1}\frac{\phi_\ell(x)}{\lambda^\ell}\biggr]\biggr|^{j}_{\leq i-1} \\
         + \frac{L_f\phi_{i-1}(x)}{\lambda^{i-1}}.
    \end{multline}
    Clearly, $L_fT_{a,1} = L_f \phi_0$
    so the property is true for $i=1$.
    Assume $\mathcal{P}(i)$ is true for a certain 
    $i\geq 1$.
    Let us show that $\mathcal{P}(i+1)$ holds. 
   We have
\begin{align*}
L_f  T_{a,i+1}(x,\lambda)
&=\sum_{\ell=0}^i \frac{L_f \phi_\ell (x)}{\lambda^\ell}\\
&=L_f T_{a,i}(x,\lambda) + \frac{L_f \phi_i(x)}{\lambda^i}\\
&=\begin{multlined}[t]
\lambda \sum_{j=1}^{i-1} \frac{\sigma^{(j)}(x)}{j!}\biggr[\sum_{\ell=1}^{i-1}\frac{\phi_\ell(x)}{\lambda^\ell}\biggr]\biggr|^j_{\leq i-1}\\
     +  \frac{L_f \phi_{i-1}(x)}{\lambda^{i-1}} + \frac{L_f \phi_i(x)}{\lambda^i}.
\end{multlined}
\end{align*}
On another hand, we have
\begin{align}  \nonumber
        & \lambda \sum_{j=1}^i \frac{\sigma^{(j)}(x)}{j!} \biggr[\sum_{\ell=1}^{i-1}\frac{\phi_\ell(x)}{\lambda^\ell} + \frac{\phi_i(x)}{\lambda^i}\biggr]\biggr|^j_{\leq i} + \frac{L_f \phi_i(x)}{\lambda^i} \nonumber
        \\
        \begin{split} \nonumber
        & = \lambda\biggr[\sum_{j=1}^{i-1} \frac{\sigma^{(j)}(x)}{j!}\biggr[\sum_{\ell=1}^{i-1} \frac{\phi_\ell(x)}{\lambda^\ell}\biggr]\biggr|^j_{\leq i} + \frac{\sigma^{(1)}(x)}{\lambda^i}\phi_i(x)  \\
        &\qquad +\frac{\sigma^{(i)}(x)}{i!}\biggr[\sum_{\ell=1}^{i-1}\frac{1}{\lambda^\ell}\phi_\ell(x)\biggr]\biggr|^i_{\leq i}\biggr]
         + \frac{L_f \phi_i(x)}{\lambda^i}
        \end{split}
        \\
        \begin{split} \nonumber
        &= \lambda\biggr[\sum_{j=1}^{i-1} \frac{\sigma^{(j)}(x)}{j!}\biggr[\sum_{\ell=1}^{i-1} \frac{\phi_\ell(x)}{\lambda^\ell}\biggr]\biggr|^j_{\leq i-1}
        \\
        &\qquad+\sum_{j=1}^{i-1} \frac{\sigma^{(j)}(x)}{j!}\biggr[\sum_{\ell=1}^{i-1} \frac{\phi_\ell(x)}{\lambda^\ell}\biggr]\biggr|^j_{= i}\\
        &\qquad + \sigma^{(1)}(x)\frac{\phi_i(x)}{\lambda^i} +\frac{1}{\lambda^i}\frac{\sigma^{(i)}(x)}{i!}(\phi_1(x))^i\biggr]
         + \frac{L_f \phi_i(x)}{\lambda^i}
        \end{split}\\
        \begin{split} \nonumber
        &= L_f  T_{a,i+1}(x,\lambda)+\lambda\biggr[ \sum_{j=1}^{i-1} \frac{\sigma^{(j)}(x)}{j!}\biggr[\sum_{\ell=1}^{i-1} \frac{\phi_\ell(x)}{\lambda^\ell}\biggr]\biggr|^j_{= i}\\
        &\qquad + \sigma^{(1)}(x)\frac{\phi_i(x)}{\lambda^i} +\frac{1}{\lambda^i}\frac{\sigma^{(i)}(x)}{i!}(\phi_1(x))^i\biggr]-\frac{L_f \phi_{i-1}(x)}{\lambda^{i-1}}.
        \end{split}
\end{align}
Therefore, with $\phi_i$ defined in \eqref{eq_phii} and noting that
$$
\biggr[\sum_{\ell=1}^{i-1} \frac{\phi_\ell(x)}{\lambda^\ell}\biggr]\biggr|^j_{= i} =\frac{1}{\lambda^i} \sum_{\substack{\ell_1+\dots+\ell_j = i\\
     1\leq \ell_1,\dots,\ell_j\leq i-1
     }}
\phi_{\ell_1}(x)\cdots\phi_{\ell_j}(x)
$$
it yields the result.

\subsection{Proof of Proposition \ref{Prop_RLipsch}}
\label{sec_Prop_RLipsch}

 To simplify the readability of this proof, in the following the dependencies on $\lambda$ have been removed.

Note that given $(z_a, r_a, y_a, z_b, r_b, y_b)$ in $\RR^6$, the function $\sigma$ being $C^2$, we have
$$
\sigma(z_a+r_a, y_a)-\sigma(z_a,y_a) = 
\left[\int_0^1 \frac{\partial\sigma}{\partial z}(z_a+\theta r_a,y_a)\dd \theta\right]r_a
$$
and
\begin{multline*}
    (\sigma(z_a+r_a, y_a)-\sigma(z_a, y_a)) -(\sigma(z_b+r_b, y_b)-\sigma(z_b,y_b))\\
    \begin{aligned}[t]
    &\begin{multlined}[t]
    =\left[\int_0^1 \frac{\partial\sigma}{\partial z}(z_a+\theta r_a,y_a)\dd \theta\right]r_a
    - \left[\int_0^1 \frac{\partial\sigma}{\partial z}(z_b+\theta r_b,y_b)\dd \theta\right]r_b
    \end{multlined}\\
    &\begin{multlined}[t]
    =\left[\int_0^1 \frac{\partial\sigma}{\partial z}(z_a+\theta r_a,y_a)\dd \theta\right][r_a-r_b]\\
    +\biggr[\int_0^1 \frac{\partial\sigma}{\partial z}(z_a+\theta r_a,y_a) 
    - \frac{\partial\sigma}{\partial z}(z_b+\theta r_b,y_b)\dd \theta\biggr]r_b
    \end{multlined}\\
    &\begin{multlined}[t]
    =\biggr[\int_0^1 \frac{\partial\sigma}{\partial z}(z_a+\theta r_a,y_a)\dd \theta\biggr][r_a-r_b]\\
+\biggr(\int_0^1\biggr[\int_0^1\frac{\partial^2\sigma}{\partial z^2}(z(\theta,\rho),y(\rho)) (z_a-z_b + \theta(r_a-r_b))\dd \rho\biggr]\dd \theta\biggr)r_b\\
+\biggr(\int_0^1\biggr[\int_0^1\frac{\partial^2\sigma}{\partial z\partial y}(z(\theta,\rho),y(\rho)) (y_a-y_b)\dd \rho\biggr]\dd \theta\biggr)r_b
    \end{multlined}
    \end{aligned}
\end{multline*}
where
\begin{align*}
    z(\theta,\rho) &= z_b+\theta r_b + \rho[z_a-z_b + \theta(r_a-r_b)] \\
    y(\rho) &= y_b+ \rho(y_a-y_b).
\end{align*}
Moreover, 
for all $(x_a,x_b)$  in $(\X+\delta_u)^2$,
\begin{align}
\begin{split}\nonumber
    &\frac{d}{dt}(R(X(x_a,t))-R(X(x_b,t))) \\
    &=  \lambda\bigg[\sigma(T_a(X(x_a,t))+R(X(x_a,t)),h(X(x_a,t)))\\
    &\qquad\qquad\qquad-\sigma(T_a(X(x_a,t)),h(X(x_a,t)))\\
    &\qquad-\big(\sigma(T_a(X(x_b,t))+R(X(x_b,t)),h(X(x_b,t)))\\
    &\qquad\qquad\qquad -\sigma(T_a(X(x_b,t))),h(X(x_b,t))\big)\bigg] \\
    &\qquad -\lambda\big[\omega(X(x_a,t))-\omega(X(x_b,t))\big].
\end{split}
\end{align}
On another hand, with Proposition \ref{Prop_RBounded}, there exists a positive real number $B_{\sigma,2}$ such that for all $(x_a,x_b)$ in $(\X+\delta_u)^2$ and for all $\lambda>1$, and $(\theta,\rho)$ in $[0,1]^2$,
\begin{align*} 
\biggr|\cfrac{\partial^2\sigma}{\partial z^2}(z(\theta,\rho),y(\rho))\biggr|&\leq B_{\sigma,2} \\
\biggr|\cfrac{\partial^2\sigma}{\partial z\partial y}(z(\theta,\rho),y(\rho))\biggr|&\leq B_{\sigma,2}
\end{align*}
with $z_a=T_a(x_a)$, $z_b=T_a(x_b)$, $r_a = R(x_a)$, $r_b = R(x_b)$, $y_a = h(x_a)$, and $y_b = h(x_b)$.
Hence, since $\X+\delta_u$ is invariant along the flow and with \eqref{eq:condsigmay}, we get
\begin{align}
\begin{split}\nonumber
        &\frac{d}{dt}\Delta R(X(x_a,t),X(x_b,t))^2 \\
        & \leq -2\lambda(\alpha-B_{\sigma,2}|R(X(x_b,t))|)\Delta R(X(x_a,t),X(x_b,t))^2 \\
        &\qquad +2\lambda B_{\sigma,2}|R(X(x_b,t))| \\
        &\qquad\qquad S(X(x_a,t),X(x_b,t))|\Delta R(X(x_a,t),X(x_b,t))|\\
        &\qquad+2\lambda|\Delta \omega(X(x_a,t),X(x_b,t))||\Delta R(X(x_a,t),X(x_b,t))|,
    \end{split}
\end{align}
where $S = |\Delta T_a| + |\Delta h|$.
Using Young's inequality twice, we obtain for some $\varepsilon$ and $\varepsilon'$ strictly positive:
\begin{align}
    \begin{split} \nonumber
        &\frac{d}{dt}\Delta R(X(x_a,t),X(x_b,t))^2 \\
        &\leq -2\lambda\biggr(\alpha-\frac{\varepsilon'}{2} - B_{\sigma,2}|R(X(x_b,t))|\biggr(1+\frac{\varepsilon}{2}\biggr)\biggr) \\
        &\qquad \times \Delta R(X(x_a,t),X(x_b,t))^2 \\
        &\qquad +\frac{\lambda}{\varepsilon}B_{\sigma,2}|R(X(x_b,t))|S(X(x_a,t),X(x_b,t))^2 \\
        &\qquad + \frac{\lambda}{\varepsilon'}\Delta \omega(X(x_a,t),X(x_b,t))^2
    \end{split}
\end{align}
Then, using Proposition \ref{Prop_RBounded},
\begin{align}
    \begin{split} \nonumber
        &\frac{d}{dt}\Delta R(X(x_a,t),X(x_b,t))^2 \\
        & \leq -2\lambda\left(\alpha-\frac{\varepsilon'}{2} - B_{\sigma,2}\frac{\bar{\omega}}{\alpha\lambda^m}\left(1+\frac{\varepsilon}{2}\right)\right)\\
        &\qquad \times \Delta R(X(x_a,t),X(x_b,t))^2  \\
        &\qquad +\frac{\bar{\omega}\lambda B_{\sigma,2}}{\alpha\varepsilon\lambda^m} S(X(x_a,t),X(x_b,t))^2 \\
        &\qquad + \frac{\lambda}{\varepsilon'}\Delta \omega(X(x_a,t),X(x_b,t))^2.
    \end{split}
\end{align}
Choosing $\varepsilon=\frac{\alpha^2}{2B_{\sigma,2}\bar{\omega}}\lambda^m$ and $\varepsilon'=\frac{\alpha}{2}$ gives
\begin{align}
    \begin{split}\nonumber
        &\frac{d}{dt}\Delta R(X(x_a,t),X(x_b,t))^2 \\
        &\leq -\frac{\lambda}{2}\left(\alpha - 4B_{\sigma,2}\frac{\bar{\omega}}{\alpha\lambda^m}\right)\Delta R(X(x_a,t),X(x_b,t))^2\\
        &\qquad +\frac{\bar{\omega}^2 B_{\sigma,2}\lambda}{\alpha^3 \lambda^{2m}}S(X(x_a,t),X(x_b,t))^2 \\
        &\qquad + \frac{2\lambda}{\alpha}\Delta \omega(X(x_a,t),X(x_b,t))^2.
    \end{split}
\end{align}
If we multiply by $\exp\left(\frac{\lambda}{2}\left(\alpha - 4B_{\sigma,2}\frac{\bar{\omega}}{\alpha\lambda^m}\right)t\right) \equiv \exp(\Theta t)$ on both sides, we obtain (by ommitting the dependency on $\lambda$ of $\Theta$ for now)
\begin{align}
\begin{split}
    &\frac{d}{dt}\biggr(\Delta R(X(x_a,t),X(x_b,t))^2\exp(\Theta t)\biggr)\nonumber \\ 
    &\leq \biggr(\frac{\bar{\omega}^2 B_{\sigma,2}\lambda}{\alpha^3 \lambda^{2m}}S(X(x_a,t),X(x_b,t))^2 \\
    & + \frac{2\lambda}{\alpha}\Delta \omega(X(x_a,t),X(x_b,t))^2\biggr)\exp(\Theta t)
\end{split}
\end{align}
We then integrate between $-\delta<0$ and $0$, for $(x_a,x_b)\in(\X+\delta_u)^2$ and $\lambda>1$,
\begin{align}
    \begin{split} \nonumber
        &\Delta R(x_a,x_b)^2  \\
        &\leq \Delta R(X(x_a,-\delta),X(x_b,-\delta))^2 \exp(-\Theta\delta) \\
        &\qquad+\int_{-\delta}^0 \frac{\bar{\omega}^2 B_{\sigma,2}\lambda}{\alpha^3 \lambda^{2m}}S(X(x_a,s),X(x_b,s))^2 \exp(\Theta s)\dd s\\
        &\qquad +\int_{-\delta}^0 \frac{2\lambda}{\alpha}\Delta \omega(X(x_a,s),X(x_b,s))^2 \exp(\Theta s)\dd s
    \end{split}
\end{align}
Since by Proposition \ref{Prop_Omega},  $s\mapsto\Delta \omega(X(x_a,s),X(x_b,s))^2$ is bounded and 
 $s\mapsto\Delta T(X(x_a,s),X(x_b,s))^2$ is also bounded,
it yields that for all $\lambda>\left(4B_{\sigma,2}\frac{\bar{\omega}}{\alpha^2}\right)^{1/m}$ the former integrals are well defined when $\delta$ goes to  $+\infty$.
Moreover, since by Proposition \ref{Prop_RBounded} and the invariance of $(\X+\delta_u)$, $\delta\mapsto\Delta R(X(x_a,-\delta),X(x_b,-\delta))^2$ is bounded
it yields
\begin{align}
    \begin{split} \nonumber
        &\Delta R(x_a,x_b)^2 \\
        &\leq \int_{-\infty}^0 \frac{\bar{\omega}^2 B_{\sigma,2}\lambda}{\alpha^3 \lambda^{2m}}S(X(x_a,s),X(x_b,s))^2 \exp(\Theta s)\dd s\\
        &\qquad +\int_{-\infty}^0 \frac{2\lambda}{\alpha}\Delta \omega(X(x_a,s),X(x_b,s))^2 \exp(\Theta s)\dd s.
    \end{split}
\end{align}
Also, the functions $h$ and $\phi_\ell$ being $C^1$, by definition of $T_a$ in \eqref{eq:defTa}, there exists $k_{S}$ such that for all $\lambda>1$ and all $(x_a,x_b)$ in $(\X+\delta_u)^2$,
$$
S(x_a,x_b)\leq k_{S}|x_a-x_b|.
$$
Using Proposition \ref{Prop_Omega} with $\lambda>\max\{\left(4B_{\sigma,2}\frac{\bar{\omega}}{\alpha^2}\right)^{1/m},1\}$
we can write
\begin{align}\label{firsteqRM}
    \begin{split}
        &\Delta R(x_a,x_b)^2 \\
        &\leq \frac{\bar{\omega}^2 B_{\sigma,2}\lambda}{\alpha^3 \lambda^{2m}}k_S^2 \int_{-\infty}^0 |X(x_a,s)-X(x_b,s)|^2 \exp(\Theta s)\dd s\\
        &\qquad + \frac{2\lambda}{\alpha}\frac{k_\omega^2}{\lambda^{2m}}\int_{-\infty}^0 |X(x_a,s) - X(x_b,s)|^2 \exp(\Theta s)\dd s.
    \end{split}
\end{align}
Since $f$ is Lipschitz on $(\X+\delta_u)$, there exists $k_f>0$ such that  for all $s<0$ and $(x_a,x_b)$ in $(\X+\delta_u)^2$,
\begin{multline*}
    |X(x_a,s)-X(x_b,s)| \\
    \leq |x_a-x_b| + \int_s^0 k_f|X(x_a,s) - X(x_b,s)|\dd s.
\end{multline*}
Grönwall's lemma gives for $t<0$,
\begin{equation} \nonumber
    |X(x_a,t) - X(x_b,t)| \leq |x_a-x_b|\exp(-k_ft)
\end{equation}
and reinjecting this in \eqref{firsteqRM},
\begin{align}
    \begin{split}\nonumber
        &\Delta R(x_a,x_b)^2 \\
        &\leq \frac{\bar{\omega}^2 B_{\sigma,2}\lambda k_S^2}{\alpha^3 \lambda^{2m}} |x_a-x_b|^2 \int_{-\infty}^0 \exp\left(\left[\Theta-2k_f\right]s\right)\dd s\\
        &\qquad + \frac{2\lambda}{\alpha}\frac{k_\omega^2}{\lambda^{2m}} |x_a-x_b|^2\int_{-\infty}^0 \exp\left(\left[\Theta-2k_f\right]s\right)\dd s.
    \end{split}
\end{align}
Let $\lambda^*>1$ be large enough such that
$$
\frac{\lambda}{2}\left(\alpha - 4B_{\sigma,2}\frac{\bar{\omega}}{\alpha\lambda^m}\right)-2k_f>0.
$$
Then, for all $\lambda > \lambda^*$, the integrals converge and we get
\begin{equation}\nonumber
    \int_{-\infty}^0 \exp\left(\left[\Theta-2k_f\right]s\right)\dd s = \frac{1}{\Theta - 2k_f}.
\end{equation}
By integrating the last identity into the previous one and taking the square root, we have for any $(x_a,x_b)\in\X^2$ and any $\lambda > \lambda^*$,
\begin{equation}\nonumber
    |\Delta R (x_a,x_b)| \leq \frac{1}{\lambda^m}\left(\frac{\bar{\omega}^2 B_{\sigma,2}\lambda k_S^2\alpha^2 + 2\lambda k_\omega^2}{\alpha^3 (\Theta - 2k_f)}\right)^{1/2}|x_a-x_b|.
\end{equation}

Which  can be rewritten as \eqref{eq_BoundDeltaR}
for some positive real number $k_R$.

\subsection{Proof of Proposition \ref{Prop_LipInjPhi}}
\label{sec_Prop_LipInjPhi}

We now establish Lipschitz injectivity of $\dphi$. 
For $\phi_0$, under Assumption \ref{ass:contraction}, we have
\begin{equation}
\label{eq:lowbounddPsi}
    \left| \frac{d\psi}{dy}(y)\right| > \left| \frac{\frac{\partial \sigma}{\partial y}(\psi(y),y)}{\frac{\partial \sigma}{\partial z}(\psi(y),y)}\right| \geq \frac{\gamma}{\beta} \equiv \mu_0 > 0 .
\end{equation}
So we deduce that
\begin{equation}
    \forall (x_a,x_b)\in\X^2,\hspace{0.3cm} |\Delta \phi_0(x_a, x_b)| \geq \mu_0 |\Delta h(x_a,x_b)|. \nonumber
\end{equation}

Then, an immediate recursion from \eqref{eq:def_phi01}-\eqref{eq_phii} shows that there exist $C^1$ functions  $P_1,\dots,P_{m-1}$ such that for all $x\in\X+\delta_u$, $\lambda>0$, and $\ell\in \{1,\ldots,m-1\}$, $\phi_\ell$ is written as
    $$
        \phi_\ell(x) = \kappa(h(x))^\ell \frac{d\psi}{dy}(h(x))L_f^\ell h(x) + P_\ell(\dH_\ell(x)).
        $$
       where
       $$
\dH_\ell = (h,\dots, L_f^{\ell-1}h).
$$

To simplify the expressions, when there is no ambiguity, we write $g(h) = g\circ h$ for the composition of functions.

For $\ell \geq 1$ and $(x_a,x_b)\in\X^2$, the  triangular inequality gives
\begin{equation}\nonumber
\begin{split}
    |\Delta \phi_\ell(x_a,x_b)| \geq \left| \Delta (\kappa^\ell (h) \frac{d\psi}{dy}(h) L_f^\ell h)(x_a, x_b))\right| \\- |\Delta P_\ell(\dH_\ell)(x_a, x_b)|.
    \end{split}
\end{equation}

The first term can be decomposed this way
\begin{align*}
    &\left|\Delta (\kappa^\ell(h) \frac{d\psi}{dy}(h)  L_f^\ell h)(x_a, x_b) \right| 
    \\
    \begin{split}
    & \geq \left|\left(\kappa^\ell \frac{d\psi}{dy}\right)(h(x_a))\Delta L_f^\ell h (x_a, x_b) \right|
    \\
    &\qquad -  \left|L_f^\ell(h(x_b))\kappa(h(x_a))^\ell \Delta \left( \frac{d\psi}{dy}\right)(h(x_a), h(x_b))\right| \\ 
    &\qquad - \left|L_f^\ell(h(x_b)) \frac{d\psi}{dy}(h(x_b)) \Delta \kappa^\ell (h(x_a), h(x_b))\right|.
    \end{split}
\end{align*}

We want to lower bound the first term, and upper bound the other two. The output $h(x)$ is bounded for $x$ in the compact $\X$, and as $h$ and $f$ are regular enough and $x$ lies in a compact, $L_f^\ell h(x)$ is bounded too. Moreover, $\kappa$ is $C^m$ and $\frac{d\psi}{dy}$ is $C^m$, therefore they are bounded and Lipschitz on the compact set $h(\X)$. It follows that $\kappa^l$ is Lipschitz on the compact set $h(\X)$. Moreover, using \eqref{eq:lowbounddPsi} and the definition of $\kappa$ with Assumption \ref{ass:contraction}, we get for all $(x_a,x_b)\in\X^2$
\begin{equation}\nonumber
\begin{split}
    \left|\Delta (\kappa^\ell(h) \frac{d\psi}{dy}(h) L_f^\ell h)(x_a, x_b)) \right| \geq 
    \frac{\mu_0}{\beta}
    |\Delta L_f^\ell h(x_a, x_b))| \\
    - \tilde{\rho}_\ell|\Delta h(x_a, x_b)|
    \end{split}
\end{equation}
for some $\tilde{\rho}_\ell \geq 0$.

Moreover, $P_\ell$ is $C^1$ on the compact set $H_\ell(\X)$, so we have 
\begin{align}\nonumber
    \forall (x_a,x_b)\in\X^2,\hspace{0.3cm}|\Delta P_\ell(H_\ell)(x_a, x_b))| \leq k_{P_\ell} |\Delta H_\ell(x_a, x_b)|
\end{align}
for some $k_{P_\ell}>0$.
So we finally get
\begin{equation}\nonumber
\begin{split}
    \forall (x_a,x_b)\in\X^2,\hspace{0.3cm}|\Delta \phi_\ell(x_a,x_b)| \geq \mu_\ell |\Delta L_f^\ell(h(x_a), h(x_b))| \\
    - \rho_\ell \sum_{i=0}^{\ell-1} |\Delta L_f^i (h(x_a), h(x_b))|
    \end{split}
\end{equation}

for some $\mu_\ell > 0$ and $\rho_\ell \geq 0$. 

To conclude the proof, we show the Lipschitz injectivity of $D\dphi$ with $D = \text{diag}(d_0, d_1, \dots, d_{m-1})$ for some $d_i > 0$ to be picked. We compute
\begin{align} \nonumber
\begin{split}
    {}& |\Delta (D\dphi)(x_a, x_b)| = \sum_{\ell=0}^{m-1} d_\ell|\Delta \phi_\ell(x_a, x_b)| \\
    & \geq \sum_{\ell=0}^{m-1} d_\ell\biggr( \mu_\ell |\Delta L_f^\ell h(x_a, x_b)| - \rho_\ell\sum_{i=0}^{\ell-1} |\Delta L_f^i h(x_a, x_b))| \biggr)  \\
    & = \sum_{\ell=0}^{m-1} (d_\ell \mu_\ell - \sum_{i=\ell+1}^{m-1} d_i \rho_i)   |\Delta L_f^\ell h (x_a, x_b))|.
    \end{split}
\end{align}
We want to fix
$d_\ell \mu_\ell - \sum_{i=\ell+1}^{m-1} d_i \rho_i = 1$ for all $\ell=0,\dots,m-1$.
This is equivalent to
$$
\begin{pmatrix}
\mu_0 & -\rho_1
&\cdots&\cdots& -\rho_{m-1}\\
0&\mu_1 & -\rho_2
&\cdots& -\rho_{m-1}\\
\vdots&\ddots&\ddots&\ddots&\vdots\\
\vdots&0&\ddots&\ddots&-\rho_{m-1}\\
0&\cdots&\cdots&0&\mu_{m-1}
\end{pmatrix}
\begin{pmatrix}
d_0\\
\vdots\\
d_{m-1}
\end{pmatrix}
=
\begin{pmatrix}
1\\
\vdots\\
1
\end{pmatrix}.
$$
Since all $\mu_i$ are non-zero, the matrix is invertible. Hence there exists a suitable choice of $d_i$. Assumption \eqref{ass:Hlinj} of differentiable observability gives for any $(x_a,x_b)\in\X^2$
\begin{equation}\nonumber
    |\Delta (D\dphi)(x_a, x_b)| \geq |\Delta \dH_m(x_a,x_b)|  \geq \underline{k}_H|x_a-x_b|.
\end{equation}
Hence, we readily get that for all $ (x_a,x_b)\in\X^2$,
$$
|\Delta \dphi (x_a, x_b)|\geq \frac{\underline k_H}{\max d_i}|x_a-x_b|.
$$
\fi

\end{document}